\documentclass[12pt,twoside]{amsart}

\usepackage{amsmath,amsthm,amsmath,amsrefs}

\usepackage{calc}
\usepackage{setspace}
\usepackage{amsbsy,amsmath,amsfonts,amsthm,amssymb,color}
\usepackage[margin=1in]{geometry}
\newtheorem{innercustomtheorem}{Theorem}
\newenvironment{customtheorem}[1]
  {\renewcommand\theinnercustomtheorem{#1}\innercustomtheorem}
  {\endinnercustomtheorem}

\theoremstyle{plain}
\newtheorem{theorem}{Theorem}[section] %[section]
\newtheorem{lemma}[theorem]{Lemma}

\newtheorem{corollary}[theorem]{Corollary}
\newtheorem{remark}[theorem]{Remark}
\newtheorem{problem}[theorem]{Problem}
\newtheorem{proposition}[theorem]{Proposition}

\theoremstyle{definition}

\usepackage{mathrsfs}

\newcommand{\bN}{\mathbb{N}}
\newcommand{\cT}{\mathcal{T}}

\newcommand{\cS}{\mathcal{S}}
\newcommand{\cB}{\mathcal{B}}

\newcommand{\cK}{\mathcal{K}}

\newcommand{\la}{\langle}
\newcommand{\ra}{\rangle}

\newcommand{\cH}{\mathcal{H}}
\newcommand{\cA}{\mathcal{A}}
\newcommand{\er}{\text{er}}
\newcommand{\el}{\text{el}}

\newcommand{\bofh}{\cB(\cH)}

\newcommand{\bC}{\mathbb{C}}

\newcommand{\cR}{\mathcal{R}}
\newcommand{\id}{\text{id}}

\newcommand{\bZ}{\mathbb{Z}}
\newcommand{\cI}{\mathcal{I}}

\newcommand{\bF}{\mathbb{F}}
\newcommand{\cW}{\mathcal{W}}

\newcommand{\qofh}{\mathcal{Q}(\mathcal{H})}

\newcommand{\env}{\text{env}}
\newcommand{\cL}{\mathcal{L}}
\newcommand{\bofl}{\mathcal{B}(\mathcal{L})}
\newcommand{\ess}{\text{ess}}

\usepackage{tikz}
\usepackage{tikz-cd}

\title[Ubiquity of Counterexamples to the Smith-Ward problem]{Ubiquity of Counterexamples to the Smith-Ward problem}

\author{Samuel J. Harris}

\address{Northern Arizona University\\
Department of Mathematics \& Statistics\\
801 S. Osborne Dr.\\
Flagstaff, AZ\\
86011 USA}
\email{samuel.harris@nau.edu}

\begin{document}

\begin{abstract}
The Smith-Ward problem about matrix ranges, posed in the 1980s, was recently resolved in the negative by Marcel Scherer \cite{Sch26} by obtaining a three-dimensional operator system $\cS \subseteq M_4(C_r^*(\bF_2))$ without the lifting property. However, this operator system must be exact. In this paper we show that, for every finitely generated $C^*$-algebra $\cA$ without the local lifting property (LLP), there exists a three-dimensional operator system $\cS \subseteq M_{n+2}(\cA)$ without the lifting property (LP), thus generalizing Scherer's result and eliminating the reliance on $\text{Ext}(\cA)$ not being a group. In particular, we prove that whenever $\cT$ is a finite-dimensional operator system without the LP, then $M_{n+2}(C_u^*(\cT))$ contains a $3$-dimensional operator system without the LP for some $n \leq 2(\dim(\cT)-1)$. In this way, we yield a plethora of counterexamples to the Smith-Ward problem. In particular, unlike Scherer's example, these three-dimensional operator systems fail both the LP and exactness. We also prove the existence of a three-dimensional operator system that detects nuclearity for unital $C^*$-algebras, strengthening previous work of Kavruk \cite{Ka15}.
\end{abstract}

\maketitle

\section{Introduction}

The Smith-Ward problem \cite{SW80} is a problem about matrix ranges of operators. Given an operator $T \in \bofh$ where $\cH$ is a Hilbert space and an $n \in \bN$, one can consider the \textbf{$n$-th matrix range} of $T$, defined as
\[ W^n(T)=\{ \Phi(T) \mid \Phi:\bofh \to M_n \text{ unital, completely positive}\}.\]
Arveson's extension theorem \cite{Ar69} shows that $W^n(T)$ depends only on $T$ and not the $C^*$-algebra containing it. As an example, the set $W^1(T)$ is the closure of the numerical range $\{ \la Tx,x\ra: x \in \cH, \, \|x\|=1\}$ of $T$. The sets $W^n(T)$ are important for many reasons, but the one most pertinent to this paper is that the operator system $\text{span}\{I,T,T^*\}$ is completely described by unital, completely positive (ucp) maps into matrix algebras by the proof of the Choi-Effros theorem (see \cite{CE77a}), which are described precisely by the sets $W^n(T)$ for $n \in \bN$.

The Smith-Ward problem, arising from the Smith-Ward theorem \cite{SW80}, concerns essential matrix ranges; that is, matrix ranges of operators in the Calkin algebra $\qofh=\bofh/\cK(\cH)$ (here $\cH$ is assumed to be infinite-dimensional and separable, and $\cK(\cH)$ denotes the compact operators). For an operator $T \in \bofh$, we let $\dot{T}$ be its quotient image in $\qofh$. The Smith-Ward theorem \cite{SW80} says that, if $T \in \bofh$ and $N \in \bN$, then there exists a compact operator $K \in \cK(\cH)$ such that $W^n(T+K)=W^n(\dot{T})$ for all $1 \leq n \leq N$.

The Smith-Ward problem \cite{SW80} asks whether, given an operator $T \in \bofh$, there exists a $K \in \cK(\cH)$ such that $W^n(T+K)=W^n(\dot{T})$ for all values of $n$. This problem was open until very recently when Scherer provided a counterexample \cite{Sch26}. Scherer's counterexample stems from a three-dimensional operator system $\cS \subseteq M_4(C_r^*(\bF_2))$ that is hyperrigid in the sense of Arveson (see \cite{Ar11} for an overview); that is, $C^*(\cS)=M_4(C_r^*(\bF_2))$ and every unital $*$-homomorphism $\pi:M_4(C_r^*(\bF_2)) \to \bofh$ has the property that there is only one ucp map $\psi:M_4(C_r^*(\bF_2)) \to \bofh$ extending $\pi_{|\cS}$. (While this is not the original form of hyperrigidity, it is equivalent to hyperrigidity for separable operator systems by \cite[Theorem~2.1]{Ar11}, and is sufficient for our purposes.) Scherer uses this hyperrigidity to show that $\cS$ does not have the lifting property--more specifically, that there is a unital $C^*$-algebra $\cA$ and an ideal $\cI \subseteq \cA$, along with a ucp map $\varphi:\cS \to \cA/\cI$, that has no ucp lift $\psi:\cS \to \cA$. (In fact, one can arrange for $\cA=\bofh$ and $\cI=\cK(\cH)$, so that $\cA/\cI=\qofh$ \cite[Proposition~7.4]{Ka14}.) One ingredient in the proof of the existence of a non-lifting ucp map is the significant result of Haagerup and Thorbj\o rnsen \cite{HT05} that the extension semigroup $\text{Ext}(C_r^*(\bF_2))$ of $C_r^*(\bF_2)$ is not a group; hence, there is an (injective) unital $*$-homomorphism $\pi:C_r^*(\bF_2) \to \qofh$ with no ucp lift to $\bofh$. 

We note that the Smith-Ward problem has a positive answer for an operator $T \in \bofh$ if and only if the associated operator system $\cS_{\dot{T}}:=\text{span}\{\dot{I},\dot{T},\dot{T}^*\} \subseteq \qofh$ has the lifting property (LP) (this is essentially the proof of \cite[Proposition~11.4]{Ka14}). Moreover, whenever $\cS$ is a three-dimensional operator system that does not have the LP, then one can find a ucp map $\varphi:\cS \to \qofh$ that does not have a ucp lift to $\bofh$ \cite[Proposition~7.4]{Ka14}. Replacing $\cS$ with $\varphi(\cS)$ yields an operator system of the form $\cS_{\dot{T}}$, and hence an operator $T$ for which the Smith-Ward problem has a negative answer. Thus, Scherer's three-dimensional operator system without the LP directly yields a counterexample to the Smith-Ward problem (see \cite{Sch26} for an argument on how to get from the operator system to the operator $T$). (We note in passing that Scherer's operator system must be three-dimensional, since all one and two-dimensional operator systems are $C^*$-nuclear and automatically satisfy the LP; see Proposition \ref{proposition: one and two dimensional systems}.)

In this paper, we generalize the hyperrigidity theorem of Scherer in the following way:

\begin{customtheorem}{\ref{theorem: S hyperrigid in matrices over A}}
(rephrased) If $\cA$ is a finitely generated unital $C^*$-algebra, then there is an $n \in \bN$ and a three-dimensional operator system $\cS \subseteq M_{n+2}(\cA)$ that is hyperrigid.
\end{customtheorem}

When $\cA=C_r^*(\bF_2)$ we have $n=2$ and the operator system $\cS$ appearing in Theorem \ref{theorem: S hyperrigid in matrices over A} is exactly the operator system in Scherer's work \cite{Sch26}. We then give a few applications of this theorem. The first is relaxing the criteria for finding counterexamples to the Smith-Ward problem:

\begin{customtheorem}{\ref{theorem: three-dimensional counterexamples to SWP}}
Let $\cA$ be a finitely generated unital $C^*$-algebra which possesses a (not necessarily injective) unital $*$-homomorphism $\pi:\cA \to \qofh$ with no ucp lift to $\bofh$. Then there is an $n \in \bN$ and a hyperrigid three-dimensional operator system $\cS \subseteq M_{n+2}(\cA)$ for which $(\id_{n+2} \otimes \pi)_{|\cS}:\cS \to M_{n+2}(\qofh)$ has no ucp lift to $M_{n+2}(\bofh)$. In particular, $\cS$ does not have the LP and hence is a counterexample to the Smith-Ward problem.
\end{customtheorem}

As a result, whenever $\cA$ is a finitely generated unital $C^*$-algebra and $\text{Ext}(\cA)$ is not a group, then for some $n \in \bN$, $M_{n+2}(\cA)$ contains a three-dimensional operator system that is hyperrigid and does not have the LP. We find this interesting since Anderson proved in 1978 that there exists a finitely generated unital $C^*$-algebra $\cA$ where $\text{Ext}(\cA)$ is not a group \cite{An78}, before the Smith-Ward problem even originated! We note that Anderson's algebra $\cA$ is generated by two unitaries (which coincidentally generate $C_r^*(\bF_2)$) and an extra projection (or equivalently, a self-adjoint unitary by replacing the projection $p$ with the self-adjoint unitary $2p-1$), so Anderson's algebra $\cA$ already yields a counterexample to the Smith-Ward problem by Theorem \ref{theorem: three-dimensional counterexamples to SWP} with $n=3$.

Next, using some facts about how hyperrigidity interacts with operator system tensor products, we show that \textit{any} finitely generated unital $C^*$-algebra $\cA$ without the local lifting property (LLP) yields a three-dimensional operator system $\cS \subseteq M_{n+2}(\cA)$ (for some $n$) that is both hyperrigid and fails the LP, thus eliminating the requirement on $\text{Ext}(\cA)$ (see Corollary \ref{corollary: non LLP yields SWP counterexample}). We also prove that, to each finite-dimensional operator system without the LP, we can associate a three-dimensional operator system $\cS$ in matrices over the universal $C^*$-algebra $C_u^*(\cT)$ that is hyperrigid and fails both the LP and exactness (see Theorem \ref{theorem: three dimensional no LP no exactness}). (In contrast, while Scherer's operator system fails the LP, it must be exact, since exactness passes to subsystems \cite[Corollary~5.8]{KPTT13} and it is a subsystem of $M_4(C_r^*(\bF_2))$, which is exact as a result of a theorem of Choi \cite{Ch79} (see also \cite{Wa94}.)

Lastly, we use Theorem \ref{theorem: three dimensional no LP no exactness} to exhibit a three-dimensional operator system that is a nuclearity detector for unital $C^*$-algebras. An operator system $\cS$ is called a \textbf{nuclearity detector} if, whenever $\cA$ is a unital $C^*$-algebra with $\cS \otimes_{\min} \cA=\cS \otimes_{\max} \cA$ (that is, there is only one operator system structure on $\cS \otimes \cA$), then $\cA$ is a nuclear $C^*$-algebra. These operator systems were popularized in a paper of Kavruk \cite{Ka15} which exhibits a family of nuclearity detectors, the smallest of which is a four-dimensional operator system known as $\cW_{3,2}$, which is contained in $M_6$. We note that all the nuclearity detectors that Kavruk gave are actually subsystems of finite-dimensional $C^*$-algebras and are duals of what are known as WEP detectors. (An operator system $\cT$ is called a \textbf{WEP detector} if, whenever $\cA$ is a unital $C^*$-algebra, then $\cT \otimes_{\min} \cA=\cT \otimes_{\max} \cA$ if and only if $\cA$ has the weak expectation property, or WEP for short.) 

Before the resolution of the Smith-Ward problem, it was not clear if there could even be any non-$C^*$-nuclear three-dimensional operator systems (if all three-dimensional operator systems were $C^*$-nuclear, then they would all have LP and the Smith-Ward problem would have a positive answer \cite{Ka14}). Now, not only do we know that this is false, but the opposite is true:

\begin{customtheorem}{\ref{theorem: three-dimensional nuclearity detector}}
There is an $n \in \bN$ with $n \leq 6$ and a three-dimensional operator system $\cS \subseteq M_{n+2}(C_u^*(\cW_{3,2}))$ with the following properties:
\begin{itemize}
\item $\cS$ is hyperrigid in $M_{n+2}(C_u^*(\cW_{3,2}))$.
\item $\cS$ does not have the LP.
\item $\cS$ is not exact.
\item $\cS$ is a nuclearity detector.
\end{itemize}
\end{customtheorem}

As this nuclearity detector is not exact, it cannot be a subsystem of a finite-dimensional $C^*$-algebra, since exactness passes to subsystems \cite[Corollary~5.8]{KPTT13}. Theorem \ref{theorem: three-dimensional nuclearity detector} further shows that, like higher-dimensional operator systems, the world of three-dimensional operator systems is far from tame, and is in fact very wild.

The rest of the paper is organized as follows. In Section \ref{section: hyperrigidity} we prove Theorem \ref{theorem: S hyperrigid in matrices over A} on the existence of three-dimensional hyperrigid operator subsystems of $M_{n+2}(\cA)$, where $\cA$ is generated by $n$ unitaries. In Section \ref{section: non LP non exact}, we prove our main results about counterexamples to the Smith-Ward problem and exhibit three-dimensional operator systems that fail both the LP and exactness. Finally, Section \ref{section: nuclearity detector} is devoted to proving the existence of a three-dimensional nuclearity detector.

\section{Hyperrigid three-dimensional operator subsystems}\label{section: hyperrigidity}

The goal of this section is to produce, for each unital $C^*$-algebra $\cA$ generated by $n$ unitaries, a three-dimensional operator system $\cS \subseteq M_{n+2}(\cA)$ so that $\cS$ is hyperrigid in $M_{n+2}(\cA)$. For this section, we fix a unital $C^*$-algebra $\cA$ that is finitely generated, and hence generated by a finite list of unitaries. We fix a finite list $u_1,...,u_n$ of unitaries that generate $\cA$ as a $C^*$-algebra and define (for convenience) $u_{n+1}=1$. We identify $M_{n+2}(\cA)$ with $M_{n+2} \otimes \cA$ throughout. Define elements $D$ and $G$ in $M_{n+2}(\cA)$ by
\begin{equation} D=\sum_{j=1}^{n+2} je_{j,j} \otimes 1 \end{equation}
\begin{equation} G=\sum_{j=1}^{n+1} (e_{1,j+1} \otimes u_i+e_{j+1,1} \otimes u_i^*)+ \sum_{j=2}^n (e_{j,j+1}+e_{j+1,j}) \otimes 1 + (e_{2,n+2}+e_{n+2,2}) \otimes 1.\end{equation}
In matrix form, the idea is that
\[ D=\begin{pmatrix} 1 & 0 & \cdots & 0 \\ 0 & 2 & \cdots & 0 \\ \vdots & \vdots & \ddots & \vdots \\ 0 & 0 & \cdots & n+2\end{pmatrix} \text{ and } G=\begin{pmatrix} 0 & u_1 & u_2 & u_3 & \cdots & u_n & 1 \\ u_1^* & 0 & 1 & 0 & \cdots & 0 &  1 \\ u_2^* & 1 & 0 & 1 & \cdots & 0 & 0 \\ u_3^* & 0 & 1 & 0 & \cdots & 0 & 0 \\ \vdots & \vdots \\ u_n^* & 0 & 0 & 0 & \cdots & 0 & 0 \\ 1 & 1 & 0 & 0 & \cdots & 0 & 0 \end{pmatrix}\]
In other words, $D$ is diagonal with entries $1,2,...,n+2$, and $G$ has the generators in the top row except entry $1$ with their adjoints in the first column, an extra $1$ in the $(2,n+2)$ and $(n+2,2)$ positions, and $1$'s on the super-diagonal and sub-diagonal from row/column $2$ to row/column $n$. This last part is so that, when $n=2$, the matrices $D$ and $G$ reduce to Scherer's matrices $D$ and $K$ \cite{Sch26}.
To streamline arguments later, we let $q_j=e_{j,j} \otimes 1$ and $s_{i,j}=q_iGq_j$ for $i,j=1,...,n+2$. Each $q_j$ is an orthogonal projection in $M_{n+2}(\cA)$ with $\sum_{j=1}^{n+2} q_j=I$. On the other hand, the block entries of $G$ are either zero or unitary, so $s_{i,j}$ is either zero, or a partial isometry with $s_{i,j}^*s_{i,j}=q_j$ and $s_{i,j}s_{i,j}^*=q_i$ in the case when $s_{i,j} \neq 0$.

We then let $\cS=\text{span}\{I,D,G\}$, which is a three-dimensional operator subsystem of $M_{n+2}(\cA)$. Our goal is to show hyperrigidity. It will follow that we must have $C^*(\cS)=M_{n+2}(\cA)$ in this setting, but it is easier to prove this first.

\begin{proposition}
\label{proposition: S generates matrices over A}
The $C^*$-algebra generated by $\cS$ is $M_{n+2}(\cA)$; that is, $C^*(\cS)=M_{n+2}(\cA)$.
\end{proposition}

\begin{proof}
Notice that the spectrum of $D$ is exactly $\{1,2,...,n+2\}$ so by the functional calculus we see that each $q_j=e_{j,j} \otimes 1$ belongs to $C^*(\cS)$. Since $u_{n+1}=1$, we see that 
\[ e_{1,n+2} \otimes 1=s_{1,n+2}=q_1Gq_{n+2} \in C^*(\cS).\]
Next, we notice that by definition of $G$, we have
\[ e_{2,n+2} \otimes 1=s_{2,n+2}=q_2Gq_{n+2} \in C^*(\cS),\]
so that
\[ e_{1,2} \otimes 1=(e_{1,n+2} \otimes 1)(e_{2,n+2} \otimes 1)^* \in C^*(\cS).\]
By definition of $G$ we have for $j=2,...,n$ that
\[ e_{j,j+1} \otimes 1=s_{j,j+1}=q_jGq_{j+1} \in C^*(\cS).\]
Thus, we have
\[ e_{1,n+1} \otimes 1=(e_{1,2} \otimes 1)(e_{2,3} \otimes 1) \cdots (e_{n,n+1} \otimes 1) \in C^*(\cS),\]
so that
\[ e_{n+1,n+2} \otimes 1=(e_{1,n+1} \otimes 1)^*(e_{1,n+2} \otimes 1) \in C^*(\cS).\]
At this point, we have shown that $e_{j,j} \otimes 1 \in C^*(\cS)$ for all $j=1,...,n+2$ and $e_{j,j+1} \in C^*(\cS)$ for all $j=1,...,n+1$, which is enough to generate all of $M_{n+2} \otimes 1$, so that $M_{n+2} \otimes 1 \in C^*(\cS)$. Since $G$ already has each of $u_1,...,u_n$ in its block entries, one can then obtain any element of the form $e_{i,j} \otimes u_k$ in $C^*(\cS)$ for $i,j=1,...,n+2$ and $k=1,...,n$. These elements generate $M_{n+2}(\cA)$ as a $C^*$-algebra, so $C^*(\cS)=M_{n+2}(\cA)$.
\end{proof}

The next couple of lemmas and the main result of this section involve Stinespring dilations. By way of notation, if $\Phi:M_{n+2}(\cA) \to \bofh$ is a ucp map, and if $\Phi=V^*\rho(\cdot)V$ is a Stinespring representation of $\Phi$ (that is, $\cL$ is a Hilbert space, $V:\cH \to \cL$ is an isometry and $\rho:M_{n+2}(\cA) \to \bofl$ is a unital $*$-homomorphism such that $\Phi(\cdot)=V^*\rho(\cdot)V$), then we will simply refer to the triple $(\rho,V,\cL)$ as a Stinespring representation of $\Phi$.

\begin{lemma}
\label{lemma: norm calculation}
Let $\pi:M_{n+2}(\cA) \to \bofh$ be an injective unital $*$-homomorphism and let $\Phi:M_{n+2}(\cA) \to \bofh$ be a ucp map extending $\pi_{|\cS}$. Let $(\rho,V,\cL)$ be a Stinespring representation of $\Phi$. Fix $1 \leq j \leq n+2$. If $V\pi(q_j)\cH \subseteq \rho(q_j)\cL$, then for $h \in \pi(q_j)\cH$ we have
\begin{equation}
\|\rho(G)Vh\|^2 = \begin{cases} (n+1)\|h\|^2 & j=1 \\ 3\|h\|^2 & 2 \leq j \leq n \\ 2\|h\|^2 & j=n+1,n+2. \end{cases}
\end{equation}
In particular, $\rho(G)Vh=V\pi(G)h$, and $\|\rho(s_{i,j})Vh\|^2=\|h\|^2$ whenever $i$ is such that $s_{i,j} \neq 0$.
\end{lemma}

\begin{proof}
Note that $\rho(s_{i,\ell})$ is either zero or a partial isometry with initial space $\rho(q_{\ell})\cL$ and final space $\rho(q_i)\cL$. Let $h \in \pi(q_j)\cH$ and assume that $V\pi(q_j)\cH \subseteq \rho(q_j)\cL$. We can write $\rho(G)Vh=\sum_{i,\ell=1}^{n+2} \rho(s_{i,\ell})Vh$. But $\rho(s_{i,\ell})Vh=0$ as long as $\ell \neq j$, which forces
\[ \rho(G)Vh=\sum_{i=1}^{n+2} \rho(s_{i,j})Vh.\]
As $\rho(s_{i,j})$ is either zero or a partial isometry with final space $\rho(q_i)$, the terms in the sum for $\rho(G)Vh$ are pairwise orthogonal since $\{\rho(q_1),...,\rho(q_{n+2})\}$ are pairwise orthogonal. Moreover, $\|\rho(s_{i,j})Vh\|^2$ is equal to $\|h\|^2$ if $\rho(s_{i,j}) \neq 0$. Considering the non-zero blocks of $G$, there are $n+1$ non-zero blocks in column $1$, and $3$ non-zero blocks in columns $2,3,...,n$, while there are $2$ non-zero blocks in columns $n+1,n+2$. Thus, of the list of elements $\{ \rho(s_{i,j})\}_{i=1}^{n+2}$, at most $n+1$ are non-zero if $j=1$; at most $3$ are non-zero if $2 \leq j \leq n$, and at most $2$ are non-zero if $j=n+1$ or $j=n+2$. So we obtain the inequality
\begin{equation}
\| \rho(G)Vh\|^2=\sum_{i=1}^{n+2} \|\rho(s_{i,j})Vh\|^2 \leq \begin{cases} (n+1)\|h\|^2 & j=1 \\ 3\|h\|^2 & 2 \leq j \leq n \\ 2\|h\|^2 & j=n+1,n+2. \end{cases} \label{equation: rho G on Vpi qj}
\end{equation}
However, applying $V^*$ yields $V^*\rho(G)Vh=\Phi(G)h=\pi(G)h$ since $G \in \cS$, so that
\begin{equation}
\|\pi(G)h\|^2 \leq \|V^*\|^2 \|\rho(G)Vh\|^2=\|\rho(G)Vh\|^2. \label{equation: compare pi(G)h norm to rho(G)Vh norm}
\end{equation}
Since $\pi$ is injective, the same argument as above yields the same claims about norm, except that they are now equalities since $\pi(s_{i,j})$ is a partial isometry whenever $s_{i,j} \neq 0$. We obtain
\begin{equation}
\|\pi(G)h\|^2=\sum_{i=1}^{n+2} \|\pi(s_{i,j})h\|^2= \begin{cases} (n+1)\|h\|^2 & j=1 \\ 3\|h\|^2 & 2 \leq j \leq n \\ 2\|h\|^2 & j=n+1,n+2. \end{cases} \label{equation: pi G on pi qj}
\end{equation}
Using (\ref{equation: compare pi(G)h norm to rho(G)Vh norm}), we see that (\ref{equation: rho G on Vpi qj}) is an equality. Thus, $\|\rho(G)Vh\|=\|\pi(G)h\|=\|V^*\rho(G)Vh\|$ whenever $h \in \pi(q_j)\cH$. Since $V^*$ preserves the norm of $\rho(G)Vh$, we must have $\rho(G)Vh \in V\cH$, so that
\[ \rho(G)Vh=VV^*\rho(G)Vh=V\pi(G)h, \, \, h \in \pi(q_j)\cH.\]
Moreover, each non-zero term in the sum in (\ref{equation: rho G on Vpi qj}) must equal $\|h\|^2$ to obtain equality, which completes the proof.
\end{proof}

\begin{lemma}
\label{lemma: moving the projections via V}
Let $\pi:M_{n+2}(\cA) \to \bofh$ be an injective unital $*$-homomorphism and let $\Phi:M_{n+2}(\cA) \to \bofh$ be a ucp map extending $\pi_{|\cS}$. Let $(\rho,V,\cL)$ be a Stinespring representation of $\Phi$. Assume that $V\pi(q_1)\cH \subseteq \rho(q_1)\cL$ and $V\pi(q_{n+2})\cH \subseteq \rho(q_{n+2})\cL$. Then $V\pi(q_j)\cH \subseteq \rho(q_j)\cL$ for all $j=1,...,n+2$.
\end{lemma}

\begin{proof}
By Lemma \ref{lemma: norm calculation}, we have $\rho(G)Vh=V\pi(G)h$ whenever $h \in \pi(q_1)\cH$ or $h \in \pi(q_{n+2})\cH$. We first show that $V\pi(q_2)\cH \subseteq \rho(q_2)\cL$. To do this, let $h \in \pi(q_{n+2})\cH$. Then
\[ \rho(q_2)\rho(G)Vh=\rho(q_2)V\pi(G)h=\rho(q_2)V\sum_{i=1}^{n+2} \pi(s_{i,n+2})h=\rho(q_2)V\sum_{i=1}^2 \pi(s_{i,n+2})h,\]
where the last equality holds since $s_{i,n+2}=0$ for $i \geq 3$. Note that since $Vh \in \rho(q_{n+2})\cL$, we have $\rho(s_{2,i})Vh=0$ if $i \neq n+2$, so can also write
\[ \rho(q_2)\rho(G)Vh=\sum_{\ell=1}^{n+2} \rho(s_{2,\ell})Vh=\rho(s_{2,n+2})Vh,\]
while the fact that $\pi(s_{1,n+2})h \in \pi(q_1)\cH$ and $V\pi(q_1)\cH\subseteq \rho(q_1)\cL$ shows that $\rho(q_2)V\pi(s_{1,n+2})h=0$, so
\[ \rho(q_2)V\sum_{i=1}^{2} \pi(s_{i,n+2})h=\rho(q_2) V \pi(s_{2,n+2})h.\]
Thus,
\[ \rho(s_{2,n+2})Vh=\rho(q_2) V  \pi(s_{2,n+2})h.\]
Since $\pi$ is injective, $s_{2,n+2}$ is a partial isometry and $V$ is an isometry, we have $\|V\pi(s_{2,n+2})h\|=\|h\|$. On the other hand, by Lemma \ref{lemma: norm calculation} with $j=n+2$ we have that $\|\rho(s_{2,n+2})Vh\|=\|h\|$. Since $\rho(q_2)$ preserves the norm of $V\pi(s_{2,n+2})h$, it follows that $V\pi(s_{2,n+2})h \in \rho(q_2)\cL$. By injectivity of $\pi$, $\pi(s_{2,n+2})$ sends the subspace $\pi(q_{n+2})\cH$ onto $\pi(q_2)\cH$, so it follows that
\[ V \pi(q_2)\cH \subseteq \rho(q_2)\cL.\]

Now, let $k \in \{3,...,n+1\}$ and suppose that we know that $V\pi(q_j)\cH \subseteq \rho(q_j)\cL$ for all $2 \leq j \leq k-1$; we will show that $V\pi(q_k)\cH \subseteq \rho(q_k)\cL$. To do this, let $h \in \pi(q_{k-1})\cH$. Since we already know that $V\pi(q_{k-1})\cH \subseteq \rho(q_{k-1})\cL$, we can use Lemma \ref{lemma: norm calculation} to conclude that $\rho(G)Vh=V\pi(G)h$. Then
\[ \rho(q_k)\rho(G)Vh=\rho(q_k)V\pi(G)h=\rho(q_k)V\sum_{i=1}^{n+2} \pi(s_{i,k-1})h.\]
Since $3 \leq k \leq n+1$, there are exactly three choices of $i$ for which $s_{i,k-1} \neq 0$: if $k=3$ then these are $i=1$, $i=3$, $i=n+2$, and if $4 \leq k \leq n+1$ then these are $i=1$, $i=k-2$ and $i=k$. Note that since $Vh \in \rho(q_{k-1})\cL$, we have $\rho(s_{i,j})Vh=0$ if $j \neq k-1$, so can also write
\[ \rho(q_k)\rho(G)Vh=\sum_{i=1}^{n+2} \rho(q_k)\rho(s_{i,k-1})Vh=\rho(s_{k,k-1})Vh.\]
Next, considering the indices $i$ for which $s_{i,k-1} \neq 0$ as before, the only indices $i$ that occur in the sum defining $\rho(q_k)V\pi(G)h$ are either those for which we already have $V\pi(q_i)\cH \subseteq \rho(q_i)\cL$, or $i=k$, so we obtain
\[ \rho(q_k)V\sum_{i=1}^{n+2} \pi(s_{i,k-1})h=\rho(q_k)V\pi(s_{k,k-1})h.\]
It follows that
\[ \rho(s_{k,k-1})Vh=\rho(q_k)V\pi(s_{k,k-1})h.\]
Since we assumed that $V\pi(q_{k-1})\cH \subseteq \rho(q_{k-1})\cL$, by Lemma \ref{lemma: norm calculation}, since $V$ is an isometry and $Vh \in V\pi(q_{k-1})\cH \subseteq \rho(q_{k-1})\cL$, we must have $\|\rho(s_{k,k-1})Vh\|=\|h\|$, while $\|V \pi(s_{k,k-1})h\|=\|h\|$ by injectivity of $\pi$. Since $\rho(q_k)$ does not decrease the norm of $V\pi(s_{k,k-1})h$, we must have $V\pi(s_{k,k-1})h \in \rho(q_k)\cL$. As $\pi(s_{k,k-1})$ maps the subspace $\pi(q_{k-1})\cH$ onto $\pi(q_k)\cH$, we then see that
\[ V\pi(q_k)\cH \subseteq \rho(q_k)\cL,\]
as desired. By induction, this result holds for all $k=3,...,n+1$, and hence for all $k=1,...,n+2$ by the assumption in the lemma statement.
\end{proof}

\begin{theorem}
\label{theorem: S hyperrigid in matrices over A}
Whenever $\pi:M_{n+2}(\cA) \to \bofh$ is an injective unital $*$-homomorphism and $\Phi:M_{n+2}(\cA) \to \bofh$ is a ucp map extending $\pi_{|\cS}$, then $\Phi=\pi$. In particular, $\cS$ is hyperrigid in $M_{n+2}(\cA)$.
\end{theorem}

\begin{proof}
If every injective unital $*$-homomorphism $\pi:M_{n+2}(\cA) \to \bofh$ is such that $\pi_{|\cS}$ has a unique extension to $M_{n+2}(\cA)$, then a theorem of Arveson \cite{Ar11} readily implies that $\cS$ has the unique extension property for all unital $*$-homomorphisms of $M_{n+2}(\cA)$, which implies that $\cS$ is hyperrigid in $M_{n+2}(\cA)$.

Thus, we proceed by letting $\pi:M_{n+2}(\cA) \to \bofh$ be an injective unital $*$-homomorphism and letting $\Phi$ be a ucp map from $M_{n+2}(\cA)$ to $\bofh$ extending $\pi_{|\cS}$. We let $(\rho,V,\cL)$ be a Stinespring representation of $\Phi$. We have
\[ \pi(D)=\sum_{j=1}^{n+2} j \pi(q_j) \text{ and } \rho(D)=\sum_{j=1}^{n+2} j\rho(q_j).\]

Now, we first show that $V\pi(q_1)\cH \subseteq \rho(q_1)\cL$ and $V\pi(q_{n+2})\cH \subseteq \rho(q_{n+2})\cL$. We start by letting $h \in \pi(q_1)\cH$; then $\pi(q_j)h=0$ for each $j=2,...,n+2$. Since $\pi(D)-I_{\cH}=\Phi(D)-I_{\cH}=\sum_{j=2}^{n+2} (j-1)\pi(q_j)$, we have
\[\la (\pi(D)-I_{\cH})h,h \ra=\la (\Phi(D)-I_{\cH})h,h \ra=\la V^*(\rho(D)-I_{\cL})Vh,h \ra=\la (\rho(D)-I_{\cL})Vh,Vh \ra.\]

We note that $D-I \geq 0$ so $\rho(D)-I_{\cL} \geq 0$. It follows that $(\rho(D)-I_{\cL})^{\frac{1}{2}} Vh=0$, so that $(\rho(D)-I_{\cL})Vh=0$. Thus, $V\pi(q_1)\cH \subseteq \ker(\rho(D)-I_{\cL})$. On the other hand, for a vector $x \in \cL$, we have $x \in \ker(\rho(D)-I_{\cL})$ if and only if $\la (\rho(D)-I_{\cL})x,x \ra=0$ since $\rho(D)-I_{\cL}$ is positive. But this is equivalent to having
\[ 0=\la (\rho(D)-I_{\cL})x,x \ra=\sum_{j=2}^{n+2} (j-1) \la \rho(q_j)x,x \ra.\]
As each $\la \rho(q_j)x,x \ra$ is non-negative, we see that $x \in \ker(\rho(D)-I_{\cL})$ if and only if $\la \rho(q_j)x,x \ra=0$ for all $j=2,...,n+2$, which is if and only if $\rho(q_j)x=0$ for all $j=2,...,n+2$. Since $\sum_{j=1}^{n+2} \rho(q_j)=I_{\cL}$, we see that $x \in \ker(\rho(D)-I_{\cL})$ if and only if $x=\rho(q_1)x \in \rho(q_1)\cL$. It follows that
\begin{equation} V\pi(q_1)\cH \subseteq \ker(\rho(D)-I_{\cL})=\rho(q_1)\cL.
\label{equation: moving q1} \end{equation}
By an identical argument, since $D \leq (n+2)I$, one can show that if $h \in \pi(q_{n+2})\cH$, then
\[ 0=\left\la \sum_{j=1}^{n+1} (n+2-j)\pi(q_j)h,h \right\ra=\la ((n+2)I_{\cL}-\rho(D))Vh,Vh \ra\]
so that 
\begin{equation} V\pi(q_{n+2})\cH \subseteq \ker((n+2)I_{\cL}-\rho(D))=\rho(q_{n+2})\cL. \label{equation: moving q(n+2)}
\end{equation}

Thus, applying Lemma \ref{lemma: moving the projections via V}, we have that $V\pi(q_j)\cH \subseteq \rho(q_j)\cL$ for all $j=1,...,n+2$. Thus, for each $j=1,...,n+2$, $\rho(D)$ scales $V\pi(q_j)$ by $j$; that is, $\rho(D)V\pi(q_j)=jV\pi(q_j)$. Using the fact that $V^*\rho(D)V=\Phi(D)=\pi(D)$, we have
\begin{align*}
\rho(D)V&=VV^*\rho(D)V+(I-VV^*)\rho(D)V \\
&=V\pi(D)+(I-VV^*)\sum_{j=1}^{n+2} \rho(D)V\pi(q_j) \\
&=V\pi(D)+(I-VV^*)\sum_{j=1}^{n+2} jV\pi(q_j) \\
&=V\pi(D)
\end{align*}
where the last step follows since $V^*V=I$ forces $(I-VV^*)V=0$. Next, if $h \in \cH$, then using the fact that $V\pi(q_j)\cH \subseteq \rho(q_j)\cL$ for each $j$, we have $\rho(G)V\pi(q_j)h=V\pi(G)\pi(q_j)h$ for each $j$ by Lemma \ref{lemma: norm calculation}, so we have
\[ \rho(G)Vh=\sum_{j=1}^{n+2} \rho(G)V\pi(q_j)h=\sum_{j=1}^{n+2} V\pi(G) \pi(q_j)h=V\pi(G)h.\]
This shows that the subspace $V\cH$ of $\cL$ is invariant for both $\rho(D)$ and $\rho(G)$. These operators are self-adjoint, so $V\cH$ is reducing for $\rho(D)$ and $\rho(G)$. As $M_{n+2}(\cA)$ is generated as a $C^*$-algebra by $D$ and $G$, $V$ is reducing for $\rho(M_{n+2}(\cA))$, making $\Phi$ a $*$-homomorphism. As $\Phi$ agrees with $\pi$ on $\cS$ and $\cS$ generates $M_{n+2}(\cA)$, we have $\Phi=\pi$, as desired.
\end{proof}

\section{Three-dimensional operator systems with neither LP nor exactness}\label{section: non LP non exact}

In this section, we prove that any finitely generated unital $C^*$-algebra with a non-lifting unital $*$-homomorphism into $\qofh$ yields a three-dimensional subsystem $\cS \subseteq M_{n+2}(\cA)$, for some $n$, without the lifting property (Theorem \ref{theorem: three-dimensional counterexamples to SWP}). Then, using how hyperrigidity interacts with certain operator system tensor products, we relax the assumption on $\cA$ to the assumption that $\cA$ is finitely generated and does not have the LLP (Corollary \ref{corollary: non LLP yields SWP counterexample}). Lastly, we provide a ubiquity of three-dimensional operator systems that fail both exactness and the LP (Theorem \ref{theorem: three dimensional no LP no exactness}). All of these operator systems yield counterexamples to the Smith-Ward problem.

While we will have more general results about counterexamples to the Smith-Ward problem than Theorem \ref{theorem: three-dimensional counterexamples to SWP} later in this section, we feel that the following theorem is important enough to state first.

\begin{theorem}
\label{theorem: three-dimensional counterexamples to SWP}
Let $\cA$ be a finitely generated unital $C^*$-algebra which possesses a (not necessarily injective) unital $*$-homomorphism $\pi:\cA \to \qofh$ with no ucp lift to $\bofh$. Then there is an $n \in \bN$ and a hyperrigid three-dimensional operator system $\cS \subseteq M_{n+2}(\cA)$ for which $(\id_{n+2} \otimes \pi)_{|\cS}:\cS \to M_{n+2}(\qofh)$ has no ucp lift to $M_{n+2}(\bofh)$. In particular, $\cS$ does not have the LP and hence is a counterexample to the Smith-Ward problem.
\end{theorem}

\begin{proof}
Suppose $\cA$ has a generating set of $n$ unitaries. We let $\cS \subseteq M_{n+2}(\cA)$ be the three-dimensional hyperrigid operator system from Theorem \ref{theorem: S hyperrigid in matrices over A}. Let $\rho:\bofh \to \qofh$ be the canonical quotient map and suppose that $\Psi:\cS \to M_{n+2}(\bofh)$ is a ucp map such that $(\id_{n+2} \otimes \rho) \circ \Psi=\varphi$, where $\varphi=(\id_{n+2} \otimes \pi)_{|\cS}$. Then by Arveson's extension theorem, we can extend $\Psi$ to a ucp map $\widetilde{\Psi}:M_{n+2}(\cA) \to M_{n+2}(\bofh)$. Then the map $\Phi=(\id_{n+2} \otimes \rho) \circ \widetilde{\Psi}:M_{n+2}(\cA) \to M_{n+2}(\qofh)$ is equal to $\varphi$ on $\cS$. By hyperrigidity, $\Phi=\id_{n+2} \otimes \pi$. But then the mapping $a \mapsto (e_{1,1} \otimes 1)\widetilde{\Psi}(I_{n+2} \otimes a)(e_{1,1} \otimes 1)$ is a ucp lift of $\pi$ to $\bofh$, a contradiction. Thus, $\cS$ does not have the LP.
\end{proof}

One immediate advantage of Corollary \ref{theorem: three-dimensional counterexamples to SWP} is that it avoids reliance on $\text{Ext}(C_r^*(\bF_2))$ not being a group, which is a significant result of Haagerup and Thorbj\o rnsen \cite{HT05}. Instead, one can appeal to a theorem of Anderson \cite{An78} from 1978 (the existence of a finitely unital $C^*$-algebra for which $\text{Ext}(\cA)$ is not a group). Interestingly, Anderson's theorem appeared before the Smith-Ward problem even originated from the paper of Smith and Ward \cite{SW80} in 1980.

We can also use Theorem \ref{theorem: S hyperrigid in matrices over A} to obtain counterexamples to the Smith-Ward problem with more peculiar properties--for example, a three-dimensional operator system that neither has LP nor is exact. Scherer's counterexample \cite{Sch26}, while not having LP, is a subsystem of $M_4(C_r^*(\bF_2))$. As $C_r^*(\bF_2)$ is exact \cite{Wa94}, so is $M_4(C_r^*(\bF_2))$, and exactness passes to operator subsystems \cite[Corollary~5.8]{KPTT13}, so Scherer's counterexample is exact. For that three-dimensional operator system $\cS$, Scherer's work shows that $\cS$ fails the LP \cite{Sch26}, while a theorem of Kavruk \cite[Theorem~6.6]{Ka14} proves that the operator system dual $\cS^d$ is not exact (but $\cS^d$ has LP since $\cS$ is exact).

Our main tool for finding a three-dimensional operator system that fails both the LP and exactness is the universal $C^*$-algebra of an operator system. Before exploring this avenue in earnest, it helps to note the low-dimensional cases first.  It is well-known that $\bC$ and $\bC^2$ are the only two operator systems of dimension $1$ and $2$, respectively, up to unital complete order isomorphism. For the sake of convenience to the reader, we include a sketch of the argument. We thank David Blecher for pointing this out to us.

\begin{proposition}
\label{proposition: one and two dimensional systems}
The only operator systems of dimension at most two, up to unital complete order isomorphism, are $\bC$ and $\bC^2$.
\end{proposition}

\begin{proof}
The one-dimensional case is trivial so we assume that $\cS\subseteq \bofh$ is a two-dimensional operator system. Then $\cS=\text{span}\{I,H\}$ for some self-adjoint contraction $H$ with $\|H\|=1$ and $\{I,H\}$ linearly independent. Then $C^*(\cS)=C^*(H) \simeq C(\sigma(H))$, the continuous functions on the spectrum of $H$, via the Gelfand transform. This mapping sends $I$ to the constant function $1$ and $H$ to the function $t$ on $\sigma(H)$. As $\cS$ is a subsystem of an abelian $C^*$-algebra, $\cS$ is a minimal operator system, in the sense that $\cS=\text{OMIN}(V)$ for some Archimedean order unit (AOU) space $V$ that is two-dimensional \cite[Theorem~3.4]{PTT11}. So we need only determine the structure of positive elements in $\cS$ rather than in matrices over $\cS$. A function $\alpha 1 + \beta t$ is positive in $\cS \subseteq C(\sigma(H))$ if and only if $\alpha+\beta t \geq 0$ for all $t \in \sigma(H)$. This inequality only depends on the extreme values for $t \in \sigma(H)$ (that is, when $|t|$ is maximized), so we may assume that $\sigma(H)=\{r,s\}$ where $-1 \leq r<s \leq 1$ and still have the same AOU space (and hence, the same minimal operator system up to unital complete order isomorphism). Note that since $\|H\|=1$, at least one of $r=-1$ or $s=1$ holds. But then it is easy to see that $C(\sigma(H))=\bC^2$, and since $\cS \subseteq C(\sigma(H))=\bC^2$ we must have $\cS=\bC^2$.  Thus, the only operator system of dimension $2$, up to complete order isomorphism, is $\bC^2$. 
\end{proof}

Thus, if $\cS$ is a one-dimensional or two-dimensional system, then it is $C^*$-nuclear and hence has the LP (this latter fact also appeared in \cite[Proposition~6.13]{Ka14}). In stark contrast, as soon as $\cT$ is an operator system with $\dim(\cT) \geq 3$, the universal $C^*$-algebra $C_u^*(\cT)$ is not exact \cite[Proposition~6.13]{Ka14}, which uses a result of Kirchberg and Wassermann that $C_u^*(\bC^3)$ is not exact \cite{KW98}. So, if we start with a finite-dimensional operator system $\cT$ without the LP (necessarily, $\dim(\cT) \geq 3$), then we can obtain a $3$-dimensional hyperrigid operator subsystem of $M_{n+2}(C_u^*(\cT))$ for some $n \in \bN$ that neither has the LP nor is exact (see Theorem \ref{theorem: three dimensional no LP no exactness}). This follows a similar argument to a previous work of the author \cite{Ha25}, but also relies on how hyperrigidity interacts with tensor products. 

First, we prove that for an operator system $\cS$ that is hyperrigid in a unital $C^*$-algebra $\cA$, exactness (respectively, LLP) of $\cS$ passes to $\cA$. To do this, we need to prove that hyperrigidity passes to certain tensor products. The analogous version of this for the minimal and essential operator system tensor products was proven in \cite[Proposition~6.2]{HK19}. We recall that the tensor product $\cS \otimes_{\text{el}} \cT$ of operator systems $\cS$ and $\cT$ is the operator system structure arising from the inclusion of $\cS \otimes \cT$ into $\cI(\cS) \otimes_{\max} \cT$, where $\cI(\cS)$ denotes the injective envelope of $\cS$. Similarly, the tensor product $\cS \otimes_{\er} \cT$ of operator systems $\cS$ and $\cT$ is the operator system structure arising from the inclusion of $\cS \otimes \cT$ into $\cS \otimes_{\max} \cI(\cT)$. We refer the reader to \cite{KPTT11} for details. One key property is that the $\text{el}$ tensor product is left injective, in the sense that if $\cS_1$ and $\cS_2$ are operator systems and $\cS_1 \subseteq \cS_2$, then the identity map $\text{id}:\cS_1 \otimes_{\text{el}} \cT \to \cS_2 \otimes_{\text{el}} \cT$ is a complete order embedding. (Similarly, the $\er$ tensor product is right injective.) In particular, if $\cA$ and $\cB$ are unital $C^*$-algebras and $\cA \subseteq \bofh$, then $\cA \otimes_{\text{el}} \cB$ is completely order isomorphic to the inclusion of $\cA \otimes \cB$ in $\bofh \otimes_{\max} \cB$, so that $\cA \otimes_{\text{el}} \cB$ is a $C^*$-algebra. (Similarly, if $\cB \subseteq \bofl$ for some Hilbert space $\cL$, then $\cA \otimes_{\er} \cB$ is completely order isomorphic to the inclusion of $\cA \otimes \cB$ in $\cA \otimes_{\max} \bofl$, so that $\cA \otimes_{\er} \cB$ is a $C^*$-algebra.)

\begin{proposition}
\label{proposition: hyperrigidity passing to el and er}
Let $\cA$ and $\cB$ be unital $C^*$-algebras and let $\cS \subseteq \cA$ and $\cT \subseteq \cB$ be operator systems. 
\begin{enumerate}
\item If $\cS$ is hyperrigid in $\cA$, then $\cS \otimes_{\el} \cB$ is hyperrigid in $\cA \otimes_{\el} \cB$. In particular, in this case we have $C_{\env}^*(\cS \otimes_{\el} \cB)=\cA \otimes_{\el} \cB$.
\item If $\cT$ is hyperrigid in $\cB$, then $\cA \otimes_{\er} \cT$ is hyperrigid in $\cA \otimes_{\er} \cB$. In particular, in this case we have $C_{\env}^*(\cA \otimes_{\er} \cT)=\cA \otimes_{\er} \cB$.
\end{enumerate}
\end{proposition}

\begin{proof}
We only prove (1), as the proof of (2) is similar. We note that by left injectivity $\cS \otimes_{\el} \cB$ is an operator subsystem of $\cA \otimes_{\el} \cB$. Let $\pi:\cA \otimes_{\el} \cB \to \bofh$ be a unital $*$-homomorphism and let $\varphi$ be the restriction of $\pi$ to $\cS \otimes_{\text{el}} \cB$. Let $\Phi:\cA \otimes_{\el} \cB \to \bofh$ be a ucp map extending $\varphi$. We identify $\cS$ and $\cA$ with $\cS \otimes 1$ and $\cA \otimes 1$, respectively, in $\cA \otimes_{\text{el}} \cB$, and we identify $\cB$ with $1 \otimes \cB$ in $\cA \otimes_{\text{el}} \cB$. Then $\Phi_{|\cA}$ is a ucp extension of $\pi_{|\cS}$, so by hyperrigidity of $\cS$ in $\cA$, we have $\Phi_{|\cA}=\pi_{|\cA}$. This implies that $\cA \otimes 1$ belongs to the multiplicative domain of $\Phi$. It is easy to see that $1 \otimes \cB$ is also in the multiplicative domain of $\Phi$, so for each $a \in \cA$ and $b \in \cB$ we must have
\[ \Phi(a \otimes b)=\Phi(a \otimes 1)\Phi(1 \otimes b)=\pi(a \otimes 1)\pi(1 \otimes b)=\pi(a \otimes b).\]
Extending by linearity and continuity we see that $\Phi=\pi$. Thus, $\cS \otimes_{\el} \cB$ is hyperrigid in $\cA \otimes_{\el} \cB$. The claim about the $C^*$-envelope then follows since, whenever $\mathcal{C}$ is a unital $C^*$-algebra and $\cR \subseteq \mathcal{C}$ is a hyperrigid operator system, then $C_{\env}^*(\cR)=\mathcal{C}$ (see \cite[Theorem~3.10]{HK19} for a proof).
\end{proof}

The next proposition is a slight generalization of work in \cite{KPTT13} where it is assumed that $\cS \subseteq \cA$ contains enough unitaries, which automatically forces $\cS$ to be hyperrigid in $\cA$ (see \cite[Lemma~3.11]{HK19}). 

\begin{proposition}
\label{proposition: hyperrigidity passes exactness and LLP}
Suppose that $\cS$ is an operator system contained in a unital $C^*$-algebra $\cA$ and suppose that $\cS$ is hyperrigid in $\cA$. 
\begin{enumerate}
\item If $\cS$ is exact, then $\cA$ is exact.
\item If $\cS$ has the LLP, then $\cA$ has the LLP.
\end{enumerate}
\end{proposition}

\begin{proof}
We prove (1). Assume that $\cS$ is exact, so that $\cS$ is $(\min,\el)$-nuclear by \cite[Theorem~5.7]{KPTT13}. Let $\cB$ be a unital $C^*$-algebra. By Proposition \ref{proposition: hyperrigidity passing to el and er}, since $\cS$ is hyperrigid in $\cA$, we have that  $C_{\text{env}}^*(\cS \otimes_{\text{el}} \cB)=\cA \otimes_{\text{el}} \cB$. On the other hand, $\cS \otimes_{\min} \cB$ is also hyperrigid in $\cA \otimes_{\min} \cB$ by the proof of \cite[Proposition~6.2]{HK19}, so by \cite[Theorem~3.10]{HK19}, $C_{\env}^*(\cS \otimes_{\min} \cB)=\cA \otimes_{\min}\cB$. Since $\cS$ is $(\min,\el)$-nuclear, we have $\cS \otimes_{\min} \cB=\cS \otimes_{\el} \cB$, so passing to $C^*$-envelopes shows that
\[ \cA \otimes_{\min} \cB=\cA \otimes_{\el} \cB.\]
Thus, $\cA$ is $C^*$-$(\min,\el)$-nuclear, which is equivalent to exactness by \cite[Theorem~5.7]{KPTT13}.

The proof of (2) follows similarly, using right injectivity of the $\er$ tensor product and the equivalence of the LLP, $C^*$-$(\min,\er)$-nuclearity and $(\min,\er)$-nuclearity \cite[Theorem~8.1]{KPTT13}.
\end{proof}

\begin{corollary}
\label{corollary: non LLP yields SWP counterexample}
If $\cA$ is a finitely generated unital $C^*$-algebra without the LLP, then there is an $n \in \bN$ and a three-dimensional operator subsystem $\cS$ of $M_{n+2}(\cA)$ that does not have the LP (and hence is a counterexample to the Smith-Ward problem).
\end{corollary}

\begin{proof}
There is some $n \in \bN$ for which $\cA$ is generated by $n$ unitaries. By Theorem \ref{theorem: S hyperrigid in matrices over A}, there is a hyperrigid operator system $\cS \subseteq M_{n+2}(\cA)$ with $\dim(\cS)=3$. By Proposition \ref{proposition: hyperrigidity passes exactness and LLP}, if $\cS$ had the LP then $\cA$ would have the LLP, a contradiction. Thus, $\cS$ does not have the LP.
\end{proof}

\begin{remark}
We note that the converse of (1) in Proposition \ref{proposition: hyperrigidity passes exactness and LLP} also holds, but this does not require hyperrigidity, since exactness passes to subsystems \cite[Corollary~5.8]{KPTT13}. On the other hand, the converse of (2) in Proposition \ref{proposition: hyperrigidity passes exactness and LLP} is false even with hyperrigidity, since the operator system $\cW_{3,2}$ (which is the dual of the operator system $NC(3)$ spanned by $1$ and the canonical generators in $C^*(*_3 \bZ_2)$) is hyperrigid in its $C^*$-envelope \cite[Theorem~6.11]{HK19} but does not have the LP, even though $C_{\env}^*(\cW_{3,2})$ is nuclear, since the dual operator system $NC(3)$ of $\cW_{3,2}$ is not exact (see \cite{Ka14}).
\end{remark}

We close this section with the following theorem, which further exhibits the ubiquity of non-LP three-dimensional operator systems. Something that is new, though, is that these can also fail exactness simultaneously.

\begin{theorem}
\label{theorem: three dimensional no LP no exactness}
Let $\cT$ be a finite-dimensional operator system without the LP. Then there exists an $n \in \bN$ with $n \leq 2( \dim(\cT)-1)$ and a three-dimensional operator subsystem $\cS \subseteq M_{n+2}(C_u^*(\cT))$ for which the following hold:
\begin{itemize}
\item $\cS$ is hyperrigid in $M_{n+2}(C_u^*(\cT))$.
\item $\cS$ does not have the LP.
\item $\cS$ is not exact.
\end{itemize}
\end{theorem}

\begin{proof}
Note that any self-adjoint element of a unital $C^*$-algebra is a linear combination of at most two unitaries from that $C^*$-algebra. So, starting with a basis $\{I,H_1,...,H_m\}$ for $\cT$ where $m=\dim(\cT)-1$ and each $H_i$ self-adjoint, each element of $\{H_1,...,H_m\}$ is a linear combination of at most two unitaries in $C_u^*(\cT)$, and these elements generate $C_u^*(\cT)$. Thus, $C_u^*(\cT)$ is generated by $n$ unitaries for some $n \leq 2(\dim(\cT)-1)$. Since $\cT$ does not have the LP, there is a ucp map $\varphi:\cT \to \qofh$ without a ucp lift to $\bofh$ by \cite[Proposition~7.4]{Ka14}. Thus, by the universal property of $C_u^*(\cT)$, there is a unital $*$-homomorphism $\pi:C_u^*(\cT) \to \qofh$ such that $\pi_{|\cT}=\varphi$. Since $\varphi$ has no ucp lift to $\bofh$, neither does $\pi$. Then applying Corollary \ref{theorem: three-dimensional counterexamples to SWP}, there is a three-dimensional operator system $\cS \subseteq M_{n+2}(C_u^*(\cT))$ that is hyperrigid, and in particular fails the LP and fails exactness by Proposition \ref{proposition: hyperrigidity passes exactness and LLP}, since $C_u^*(\cT)$ has neither the LLP nor exactness by the above. This completes the proof.
\end{proof}

\section{A three-dimensional nuclearity detector}\label{section: nuclearity detector}

In this last section, we use Theorem \ref{theorem: three dimensional no LP no exactness} to also prove the existence of a three-dimensional operator system $\cS$ that detects nuclearity for unital $C^*$-algebras. Given an operator system $\cR$, we say that $\cR$ is a \textbf{nuclearity detector} if, for any unital $C^*$-algebra $\cA$, one has that
\[ \cR \otimes_{\min} \cA=\cR \otimes_{\max} \cA \iff \cA \text{ is nuclear}.\]
Note that if $\cA$ is nuclear, then it is $(\min,\max)$-nuclear as an operator system \cite[Remark~5.15]{KPTT11}, so that $\cR \otimes_{\min} \cA = \cR \otimes_{\max} \cA$ automatically. Thus, we only ever worry about showing that, if $\cR \otimes_{\min} \cA=\cR \otimes_{\max} \cA$, then $\cA$ is nuclear. 

Kavruk proved \cite[Theorem~0.3]{Ka15} that the four-dimensional operator system
\[ \cW_{3,2}=\left\{ \begin{pmatrix} a & b \\ b & a \\ & & a & c \\ & & c & a \\ & & & & a & d \\ & & & & d & a \end{pmatrix}: a,b,c,d \in \bC \right\} \subseteq M_6\]
is a nuclearity detector. To work towards a three-dimensional nuclearity detector, we first note that $C_u^*(\cW_{3,2})$ enjoys this property.

\begin{proposition}
\label{proposition: universal of w32 is a nuclearity detector}
The universal $C^*$-algebra $C_u^*(\cW_{3,2})$ of $\cW_{3,2}$ is a nuclearity detector.
\end{proposition}

\begin{proof}
Suppose that $\cA$ is a unital $C^*$-algebra and that $C_u^*(\cW_{3,2}) \otimes_{\min} \cA = C_u^*(\cW_{3,2}) \otimes_{\max} \cA$. By injectivity of the minimal tensor product, $\cW_{3,2} \otimes_{\min} \cA$ is completely order isomorphic to its inclusion in $C_u^*(\cW_{3,2}) \otimes_{\min} \cA$. On the other hand, the operator system structure on $\cW_{3,2} \otimes \cA$ inside of $C_u^*(\cW_{3,2}) \otimes_{\max} \cA$ is the commuting tensor product $\cW_{3,2} \otimes_c \cA$ \cite[Proposition~3.4]{Ka14}, so this must equal the minimal tensor product. But all unital $C^*$-algebras are $(c,\max)$-nuclear \cite[Theorem~6.7]{KPTT11}, so we must have
\[ \cW_{3,2} \otimes_{\min} \cA=\cW_{3,2} \otimes_c \cA= \cW_{3,2} \otimes_{\max} \cA.\]
By Kavruk's theorem \cite[Theorem~0.3]{Ka15}, $\cA$ is nuclear.
\end{proof}

Now we get to the main result of this section:

\begin{theorem}
\label{theorem: three-dimensional nuclearity detector}
There is an $n \in \bN$ with $n \leq 6$ and a three-dimensional operator system $\cS \subseteq M_{n+2}(C_u^*(\cW_{3,2}))$ with the following properties:
\begin{itemize}
\item $\cS$ is hyperrigid in $M_{n+2}(C_u^*(\cW_{3,2}))$.
\item $\cS$ does not have the LP.
\item $\cS$ is not exact.
\item $\cS$ is a nuclearity detector.
\end{itemize}
\end{theorem}

\begin{proof}
Noting that $\dim(\cW_{3,2})=4$, the first three points about $\cS$ and the bound on $n$ are due to Theorem \ref{theorem: three dimensional no LP no exactness}. Thus, we need only show that this operator system $\cS$ is also a nuclearity detector. Let $\cA$ be any unital $C^*$-algebra. As $\cS$ is hyperrigid in $M_{n+2}(C_u^*(\cW_{3,2}))$, the proof of \cite[Proposition~6.2]{HK19} shows that $\cS \otimes_{\min} \cA \subseteq M_{n+2}(C_u^*(\cW_{3,2})) \otimes_{\min} \cA$ is hyperrigid and $\cS \otimes_{\ess} \cA \subseteq M_{n+2}(C_u^*(\cW_{3,2})) \otimes_{\max} \cA$ is hyperrigid, where for two operator systems $\cR$ and $\cT$, the tensor product $\cR \otimes_{\ess} \cT$ is the operator system structure arising from the inclusion $\cR \otimes \cT \subseteq C_{\env}^*(\cR) \otimes_{\max} C_{\env}^*(\cT)$. (Here, we are using $\cR=\cS$ and $\cT=\cA$.)

Thus, if $\cA$ is a unital $C^*$-algebra with $\cS \otimes_{\min} \cA=\cS \otimes_{\max} \cA$, then since every operator system tensor product is between $\min$ and $\max$ \cite{KPTT11}, we have 
\[ \cS \otimes_{\min} \cA=\cS \otimes_{\ess} \cA.\]
Passing to $C^*$-envelopes by hyperrigidity (\cite[Proposition~6.2]{HK19}) forces 
\[M_{n+2}(C_u^*(\cW_{3,2})) \otimes_{\min} \cA=M_{n+2}(C_u^*(\cW_{3,2})) \otimes_{\max} \cA.\] 
This readily implies that 
\[C_u^*(\cW_{3,2}) \otimes_{\min} \cA=C_u^*(\cW_{3,2}) \otimes_{\max} \cA,\] 
which implies that $\cA$ is nuclear by Proposition \ref{proposition: universal of w32 is a nuclearity detector}. We conclude that $\cS$ is a nuclearity detector.
\end{proof}

We note that Kavruk's nuclearity detector $\cW_{3,2}$ does not have LP, but it is exact since it is a subsystem of $M_6$ (which is nuclear, hence exact). So, at the expense of explicitness, Theorem \ref{theorem: three-dimensional nuclearity detector} provides a nuclearity detector of the smallest dimension possible (since all $1$ and $2$-dimensional operator systems are $C^*$-nuclear) that is also not exact.

While one could work out a generating set of at most $6$ unitaries for $C_u^*(\cW_{3,2})$ and hence a three-dimensional operator system $\cS \subseteq M_8(C_u^*(\cW_{3,2}))$ with the properties of Theorem \ref{theorem: three-dimensional nuclearity detector}, we have not done so here. We do not know whether there is a more explicit three-dimensional nuclearity detector (in particular, one that does not rely on universal $C^*$-algebras), though we suspect that one can be found.

We note that all nuclearity detectors exhibited in the work of Kavruk \cite{Ka15} are actually duals of WEP detectors for unital $C^*$-algebras. We are not sure if this holds for the three-dimensional operator system $\cS$ exhibited in Theorem \ref{theorem: three-dimensional nuclearity detector}.

\begin{problem}
\label{problem: dual of WEP detector}
Let $\cS$ be the three-dimensional nuclearity detector from Theorem \ref{theorem: three-dimensional nuclearity detector}. Is $\cS^d$ a WEP detector? That is to say, is it true that a unital $C^*$-algebra $\cA$ has the weak expectation property if and only if $\cS^d \otimes_{\min} \cA=\cS^d \otimes_{\max} \cA$?
\end{problem}

If the answer to Problem \ref{problem: dual of WEP detector} is yes, then there is a three-dimensional operator system $\cS$ (along with its dual) whose tensor products tell us that both the Smith-Ward problem and Connes' embedding problem have a negative answer. Indeed, the Smith-Ward problem having a negative answer is summarized in the fact that $\cS$ does not have LP, but this is equivalent to having $\cS \otimes_{\min} \bofh \neq \cS \otimes_{\max} \bofh$ where $\cH$ is infinite-dimensional and separable by \cite[Theorem~8.1]{KPTT13}. But by Theorem \ref{theorem: three-dimensional nuclearity detector}, the fact that $\cS \otimes \bofh$ does not have a unique operator system tensor product follows from the well-known fact that $\bofh$ is not nuclear!

Similarly, if $\cS^d$ is a WEP detector, then applying this to $C^*(\bF_2)$ (or any $C^*$-algebra in Kirchberg's reformulation of Connes' embedding problem \cite{Ki93}), the Connes embedding problem would have a positive answer if and only if $\cS^d \otimes_{\min} C^*(\bF_2)=\cS^d \otimes_{\max} C^*(\bF_2)$. Thanks to the landmark result that $\text{MIP}^*=\text{RE}$, we know that the Connes embedding problem has a negative answer \cite{JNVWY20}, and so $C^*(\bF_2)$ does not have WEP, which would force $\cS^d \otimes_{\min} C^*(\bF_2) \neq \cS^d \otimes_{\max} C^*(\bF_2)$. So it is possible that the negative solutions to both the Smith-Ward problem and Connes' embedding problem are wrapped up in the tensor products of this three-dimensional operator system $\cS$ (and its dual).

On the other hand, if the answer to Problem \ref{problem: dual of WEP detector} is no, then it would provide, to our knowledge, the first example of a finite-dimensional operator system that detects nuclearity but whose dual does not detect the WEP.

\section*{Acknowledgements}

We would like to thank David Blecher for helpful feedback on an earlier draft of this paper and for pointing us to Proposition \ref{proposition: one and two dimensional systems}. We would also like to thank Vern Paulsen and Mizanur Rahaman for helpful conversations and suggestions.

\end{document}